\documentclass[reqno]{amsart}
\usepackage{enumerate}
\usepackage{amsmath}%
\usepackage{amsfonts}%
\usepackage{amssymb}%
\usepackage{graphicx}
\usepackage{mathrsfs}
\newtheorem{theorem}{Theorem}
\theoremstyle{plain}

\newtheorem{case}{Case}
\newtheorem{claim}[theorem]{Claim}

\newtheorem{definition}[theorem]{Definition}

\newtheorem{fact}[theorem]{Fact}
\newtheorem{lemma}[theorem]{Lemma}

\newtheorem{proposition}[theorem]{Proposition}

\numberwithin{equation}{section}
\numberwithin{theorem}{section}
\def\A{\mathcal{A}}

\def\C{\mathcal{C}}

\def\F{\mathcal{F}}
\def\G{\mathcal{G}}
\def\h{\mathcal{H}}
\def\K{\mathcal{K}}
\def\R{\mathcal{R}}

\def\cP{\mathcal{P}}
\def\T{\mathcal{T}}
\def\ss{\mathcal{S}}
\def\Q{\mathcal{Q}}
\def\Y{\mathcal{Y}}
\def \a{\alpha}
\def \e{\epsilon}
\def \d{\text{deg}}
\def \r{\gamma}
\def \Cr{\text{Cross}}

\def\eg{\emph{e.g.}}

\begin{document}
\title{Minimum vertex degree threshold for $\C_4^3$-tiling}
\thanks{The second author is partially supported by NSA grant H98230-12-1-0283 and NSF grant DMS-1400073.}
\author{Jie Han}
\address[Jie Han and Yi Zhao]
{Department of Mathematics and Statistics,\newline
\indent Georgia State University, Atlanta, GA 30303}
\email[Jie Han]{jhan22@gsu.edu}
\author{Yi Zhao}
\email[Yi Zhao]{\texttt{yzhao6@gsu.edu}}%
\date{\today}
\begin{abstract}
We prove that the vertex degree threshold for tiling $\C_4^3$ (the 3-uniform hypergraph with four vertices and two triples) in a 3-uniform hypergraph on $n\in 4\mathbb N$ vertices is $\binom{n-1}2 - \binom{\frac34 n}2+\frac38n+c$, where $c=1$ if $n\in 8\mathbb N$ and $c=-\frac12$ otherwise. This result is best possible, and is one of the first results on vertex degree conditions for hypergraph tiling.
\end{abstract}

\maketitle

\section{Introduction}
Given $k\ge 2$, a $k$-uniform hypergraph (in short, \emph{$k$-graph}) consists of a vertex set $V$ and an edge set $E\subseteq \binom{V}{k}$, where every edge is a $k$-element subset of $V$. Given a $k$-graph $\h$ with a set $S$ of $d$ vertices (where $1\le d\le k-1$) we define $\d_{\h}(S)$ to be the number of edges containing $S$ (the subscript $\h$ is omitted if it is clear from the context). The \emph{minimum $d$-degree} $\delta_d(\h)$ of $\h$ is the minimum of $\d_{\h}(S)$ over all $d$-vertex sets $S$ in $\h$.

Given a $k$-graph $\G$ of order $g$ and a $k$-graph $\h$ of order $n$, a \emph{$\G$-tiling} (or \emph{$\G$-packing}) of $\h$ is a subgraph of $\h$ that consists of vertex-disjoint copies of $\G$. When $g$ divides $n$, a \emph{perfect $\G$-tiling} (or a \emph{$\G$-factor}) of $\h$ is a $\G$-tiling of $\h$ consisting of $n/g$ copies of $\G$. Define $t_d(n, \G)$ to be the smallest integer $t$ such that every $k$-graph $\h$ of order $n\in g\mathbb{N}$ with $\delta_d(\h)\ge t$ contains a perfect $\G$-tiling.

As a natural extension of the matching problem, tiling has been an active area in the past two decades (see surveys \cite{KuOs-survey,RR}). Much work has been done on the problem for graphs ($k=2$), see \eg, \cite{HaSz, AY96, KSS-AY, KuOs09}. In particular, K\"uhn and Osthus \cite{KuOs09} determined $t_1(n, \G)$, for any graph $\G$, up to an additive constant. Tiling problems become much harder for hypergraphs. For example, despite much recent progress \cite{AFHRRS, CzKa, Khan2, Khan1, KOT, RRS09, TrZh13}, we still do not know the 1-degree threshold for a perfect matching in $k$-graphs for arbitrary $k$.

Other than the matching problem, only a few tiling thresholds are known.
Let $K_4^3$ be the complete 3-graph on four vertices, and let $K_4^3-e$ be the (unique) 3-graph on four vertices with three edges. Recently Lo and Markstr\"om \cite{LM1} proved that $t_2(n, K_4^3)=(1+o(1))3n/4$, and independently Keevash and Mycroft \cite{KM1} determined the exact value of $t_2(n, K_4^3)$  for sufficiently large $n$. In \cite{LM2}, Lo and Markstr\"om proved that $t_2(n, K_4^3 - e)=(1+o(1))n/2$. Let
$\C_4^3$ be the unique 3-graph on four vertices with two edges. This 3-graph was denoted by $K_4^3 - 2e$ in \cite{CDN}, and by $\Y$ in \cite{HZ1}. Here we follow the notation in \cite{KO} and view it as a \emph{cycle} on four vertices.  K\"uhn and Osthus \cite{KO} showed that $t_2(n, \C_4^3)=(1+o(1))n/4$, and Czygrinow, DeBiasio and Nagle \cite{CDN} recently determined $t_2(n, \C_4^3)$ exactly for large $n$. In this paper we determine $t_1(n, \C_4^3)$ for sufficiently large $n$. From now on, we simply write $\C_4^3$ as $\C$.

Previously we only knew $t_1(n, K_3^3)$ \cite{Khan1,KOT} and $t_1(n, K_4^4)$ \cite{Khan2} exactly, and $t_1(n, K_5^5)$ \cite{AFHRRS}, $t_1(n, K_3^3(m))$ and $t_1(n, K_4^4(m))$ \cite{LM1} asymptotically,
where $K_k^k$ denotes a single $k$-edge, and $K_k^k(m)$ denotes the complete $k$-partite $k$-graph with $m$ vertices in each part. So Theorem~\ref{thmmain} below is one of the first (exact) results on vertex degree conditions for hypergraph tiling.

\begin{theorem}[Main Result] \label{thmmain}
Suppose $\h$ is a 3-graph on $n$ vertices such that $n\in 4\mathbb N$ is sufficiently large and
\begin{equation}
\delta_1(\h)\ge \binom{n-1}2 - \binom{\frac34 n}2+\frac{3}8n+c(n), \label{eqdeg}
\end{equation}
where $c(n)=1$ if $n\in 8\mathbb N$ and $c(n)=-1/2$ otherwise. Then $\h$ contains a perfect $\C$-tiling.
\end{theorem}

Proposition~\ref{factcounter} below shows that Theorem~\ref{thmmain} is best possible. Theorem~\ref{thmmain} and Proposition~\ref{factcounter} together imply that $t_1(n, \C)=\binom{n-1}2 - \binom{\frac34 n}2+\frac{3}8n+c(n)$.

\begin{proposition}\label{factcounter}
For every $n\in 4\mathbb{N}$ there exists a 3-graph of order $n$ with minimum vertex degree $\binom{n-1}2 - \binom{\frac34 n}2+\frac38 n+c(n)-1$, which does not contain a perfect $\C$-tiling.
\end{proposition}

\begin{proof}
We give two constructions similar to those in \cite{CDN}. Let $V=A\dot\cup B$
\footnote{Throughout the paper, we write $A\dot\cup B$ for $A\cup B$ when sets $A$, $B$ are disjoint.}
with $|A| =\frac n4 - 1$ and $|B|= \frac {3n}4+1$.
A Steiner system $S(2,3,m)$ is a 3-graph $\ss$ on $n$ vertices such that every pair of vertices has degree one -- so $S(2,3,m)$ contains no copy of $\C$. It is well-known that an $S(2,3,m)$ exists if and only if $m\equiv 1, 3 \mod 6$.

Let $\h_0=(V, E_0)$ be the 3-graph on $n\in 8\mathbb N$ vertices as follows. Let $E_0$ be the set of all triples intersecting $A$ plus a Steiner system $S(2,3,\frac34n+1)$ in $B$. Since for the Steiner system $S(2,3,\frac34n+1)$, each vertex is in exactly $\frac34n/2=\frac 38n$ edges, we have $\delta_1(\h_0)= \binom{n-1}2 - \binom{\frac {3n}4}2+\frac38 n$. Furthermore, since $B$ contains no copy of $\C$, the size of the largest $\C$-tiling in $\h_0$ is $|A|=\frac n4-1$. So $\h_0$ does not contain a perfect $\C$-tiling.

On the other hand, let $\h_1=(V, E_1)$ be the 3-graph on $n\in 4\mathbb N\setminus 8\mathbb N$ vertices as follows. Let $\G$ be a Steiner system of order $\frac 34n+4$. This is possible since $\frac34n+4\equiv 1\mod 6$. Then pick an edge $abc$ in $\G$ and let $\G'$ be the induced subgraph of $\G$ on $V(\G)\setminus\{a,b,c\}$. Finally let $E_1$ be the set of all triples intersecting $A$ plus $\G'$ induced on $B$. Since $\G$ is a regular graph with vertex degree $\frac12(\frac34n+4-1)=\frac38n+\frac32$, we have that $\delta_1(\G')=\frac38n+\frac32-3=\frac38n-\frac32$. Thus, $\delta_1(\h_1)= \binom{n-1}2 - \binom{\frac {3n}4}2+\frac38 n -\frac32$. As in the previous case, $\h_1$ does not contain a perfect $\C$-tiling.
\end{proof}

As a typical approach of obtaining exact results, we distinguish the \emph{extremal} case from the \emph{nonextremal} case and solve them separately. Given a 3-graph $\h$ of order $n$, we say that $\h$ is $\C$-free if $\h$ contains no copy of $\C$. In this case, clearly, every pair of vertices has degree at most one. Every vertex has degree at most $\frac{n-1}2$ because its link graph\footnote{Given 3-graph $\h=(V, E)$ and $x\in V$, the link graph of $x$ has vertex set $V\setminus \{x\}$ and the edge set $\{e\setminus \{x\}:e\in E(\h), x\in e\}$.
} contains no vertex of degree two.

\begin{definition}
Given $\e>0$, a 3-graph $\h$ on $n$ vertices is called $\e$-extremal if there is a set $S\subseteq V(\h)$, such that $|S| \ge (1-\e)\frac{3n}4$ and $\h[S]$ is $\C$-free.
\end{definition}

\begin{theorem}[Extremal Case]\label {lemE}
There exists ${\e}>0$ such that for every 3-graph $\h$ on $n$ vertices, where $n\in 4\mathbb N$ is sufficiently large, if $\h$ is ${\e}$-extremal and satisfies \eqref{eqdeg}, then $\h$ contains a perfect $\C$-tiling.
\end{theorem}

\begin{theorem}[Nonextremal Case]\label {lemNE}
For any ${\e}>0$, there exists $\r>0$ such that the following holds. Let $\h$ be a 3-graph on $n$ vertices, where $n\in 4\mathbb N$ is sufficiently large. If $\h$ is not ${\e}$-extremal and satisfies $\delta_1(\h)\ge \left(\frac7{16} - \r \right)\binom n {2}$, then $\h$ contains a perfect $\C$-tiling.
\end{theorem}

Theorem \ref{thmmain} follows Theorems \ref{lemE} and \ref{lemNE} immediately by choosing $\e$ from Theorem~\ref{lemE}. The proof of Theorem~\ref{lemE} is somewhat routine and will be presented in in Section 4.

The proof of Theorem \ref{lemNE}, as the one of \cite[Theorem 1.5]{CDN}, uses the \emph{absorbing method} initiated by R\"odl, Ruci\'nski and Szemer\'edi, \eg, \cite{RRS06, RRS08}. More precisely, we find the perfect $\C$-tiling by applying the Absorbing Lemma below and the $\C$-tiling Lemma \cite[Lemma 2.15]{HZ1} together.

\begin{lemma}[Absorbing Lemma]\label{lemA}
For any $0<\theta\le 10^{-4}$, there exist $\beta>0$ and integer $n_{\ref{lemA}}$ such that the following holds. Let $\h$ be a 3-graph of order $n\ge n_{\ref{lemA}}$ with $\delta_1(\h)\ge (\frac14+\theta)\binom n2$. Then there is a vertex set $W\in V(\h)$ with $|W|\in 4\mathbb N$ and $|W|\le 2049\theta n$ such that for any vertex subset $U$ with $U\cap W=\emptyset$, $|U|\in 4\mathbb N$ and $|U|\le \beta n$ both $\h[W]$ and $\h[W\cup U]$ contain $\C$-factors.
\end{lemma}

\begin{lemma}[$\C$-tiling Lemma, \cite{HZ1}] \label{lemM}
For any $0<\r <1$, there exists an integer $n_{\ref{lemM}}$ such that the following holds. Suppose $\h$ is a 3-graph on $n>n_{\ref{lemM}}$ vertices with
\[
\delta_1(\h)\ge \left(\frac7{16} - \r \right)\binom n {2},
\]
then $\h$ contains a $\C$-tiling covering all but at most $2^{19}/\r$ vertices or $\h$ is $2^{11}\r$-extremal.
\end{lemma}

\smallskip
We postpone the proof of Lemma \ref{lemA} to Section 3 and prove Theorem \ref{lemNE} now.

\begin{proof}[Proof of Theorem \ref{lemNE}]

Without loss of generality, assume $0< \e < 1$. Let $\r=2^{-13}{\e}$ and $\theta=10^{-4}\r$ (thus $\theta<10^{-4}$). We find $\beta$ by applying Lemma \ref{lemA}. Choose $n\in 4\mathbb N$ such that $n>\max\{n_{\ref{lemA}}, 2n_{\ref{lemM}}, 2^{18}/(\r\beta)\}$.
Let $\h=(V, E)$ be a 3-graph on $n$ vertices. Suppose that $\h$ is not ${\e}$-extremal and $\delta_1(\h)\ge \left(\frac7{16} - \r \right)\binom n {2}$. First we apply Lemma \ref{lemA} to $\h$ and find the absorbing set $W$ with $|W|\le 2049\theta n$. Let $\h'=\h[V\setminus W]$ and $n'=n-|W|$. Note that $2|W| < 10^4 \theta n = \r n$ and thus $n' > n - \r n/2 > n_{\ref{lemM}}$. Furthermore,
\[
\delta_1(\h')\ge \delta_1(\h) - |W|(n-1)\ge \left(\frac7{16} - 2\r \right)\binom {n} {2} \ge \left(\frac7{16} - 2\r \right)\binom {n'} {2}.
\]
Second we apply Lemma \ref{lemM} to $\h'$ with parameter $2\r$ in place of $\r$ and derive that either $\h'$ is $2^{12}\r$-extremal or $\h'$ contains a $\C$-tiling covering all but at most $2^{18}/\r$ vertices. In the former case, since
\[
(1-2^{12}\r)\frac{3n'}4 > (1-2^{12}\r)\frac{3}{4}\left(n- \frac{\r n}2\right)>(1-2^{13}\r)\frac{3n}4=(1-{\e})\frac{3n}4,
\]
$\h$ is ${\e}$-extremal, a contradiction. In the latter case, let $U$ be the set of uncovered vertices in $\h'$. Then we have $|U|\in 4\mathbb{N}$ and $|U|\le 2^{18}/\r\le \beta n$ by the choice of $n$. By Lemma \ref{lemA}, $\h[W\cup U]$ contains a perfect $\C$-tiling. Together with the $\C$-tiling provided by Lemma \ref{lemM}, this gives a perfect $\C$-tiling of $\h$.
\end{proof}

\medskip

The Absorbing Lemma and $\C$-tiling Lemma in \cite{CDN} are not very difficult to prove because of the co-degree condition. In contrast, our corresponding lemmas are harder. Luckily we already proved Lemma \ref{lemM} in \cite{HZ1} (as a key step for finding a loose Hamilton cycle in 3-graphs). In order to prove Lemma \ref{lemA}, we will use the Strong Regularity Lemma and an extension lemma from \cite{CFKO}, which is a corollary of the counting lemma.

The rest of the paper is organized as follows. We introduce the Regularity Lemma in Section 2, prove Lemma \ref{lemA} in Section 3, and finally prove Theorem \ref{lemE} in Section 4.

\section{Regularity Lemma for $3$-graphs}
\subsection{Regular complexes}
Before we can state the regularity lemma, we first define a complex. A hypergraph $\h$ consists of a vertex set $V(\h)$ and an edge set $E(\h)$, where every edge $e\in E(\h)$ is a non-empty subset of $V(\h)$. A hypergraph $\h$ is a \emph{complex} if whenever $e\in E(\h)$ and $e'$ is a non-empty subset of $e$ we have that $e'\in E(\h)$. All the complexes considered in this paper have the property that every vertex forms an edge.

For a positive integer $k$, a complex $\h$ is a \emph{$k$-complex} if every edge of $\h$ consists of at most $k$ vertices. The edges of size $i$ are called $i$-edges of $\h$. Given a $k$-complex $\h$, for each $i\in [k]$ we denote by $\h_i$ the underlying $i$-graph of $\h$: the vertices of $\h_i$ are those of $\h$ and the edges of $\h_i$ are the $i$-edges of $\h$.

Given $s\ge k$, a \emph{$(k,s)$-complex} $\h$ is an $s$-partite $k$-complex, by which we mean that the vertex set of $\h$ can be partitioned into sets $V_1, \dots, V_s$ such that every edge of $\h$ is \emph{crossing}, namely, meets each $V_i$ in at most one vertex.

Given $i\ge 2$, an $i$-partite $i$-graph $\h$ and an $i$-partite $(i-1)$-graph $\G$ on the same vertex set, we write $\K_i(\G)$ for the family of all crossing $i$-sets that form a copy of the complete $(i-1)$-graph $K_i^{(i-1)}$ in $\G$. We define the density of $\h$ with respect to $\G$ to be
\[
d(\h|\G):=\frac{|\K_i(\G)\cap E(\h)|}{|\K_i(\G)|} \quad \text{if} \quad |\K_i(\G)|>0,
\]
and $d(\h|\G)=0$ otherwise. More generally, if ${\bf Q}=(Q_1, \dots, Q_r)$ is a collection of $r$ subhypergraphs of $\G$, we define $\K_i({\bf Q}):=\bigcup_{j=1}^r \K_i(Q_j)$ and
\[
d(\h|{\bf Q}):=\frac{|\K_i({\bf Q})\cap E(\h)|}{|\K_i({\bf Q})|} \quad \text{if} \quad |\K_i({\bf Q})|>0,
\]
and $d(\h|{\bf Q})=0$ otherwise.

We say that $\h$ is \emph{$(d,\delta,r)$-regular with respect to $\G$} if every $r$-tuple ${\bf Q}$ with $|\K_i({\bf Q})|>\delta |\K_i(\G)|$ satisfies $|d(\h|{\bf Q})-d|\le \delta$. Instead of $(d, \delta, 1)$-regularity we simply refer to \emph{$(d, \delta)$-regularity}.

Given a $(3,3)$-complex $\h$, we say that \emph{$\h$ is $(d_3,d_2,\delta_3,\delta,r)$-regular} if the following conditions hold:
\begin{enumerate}
\item For every pair $K$ of vertex classes, $\h_2[K]$ is $(d_2,\delta)$-regular with respect to $\h_1[K]$ unless $e(\h_2[K])=0$, where $\h_i[K]$ is the restriction of $\h_i$ to the union of all vertex classes in $K$.

\item $\h_3$ is $(d_3, \delta_3, r)$-regular with respect to $\h_2$ unless $e(\h_3)=0$.
\end{enumerate}

\subsection{Statement of the Regularity Lemma}

In this section we state the version of the regularity lemma due to R\"odl and Schacht \cite{RS} for 3-graphs, which is almost the same as the one given by Frankl and R\"odl \cite{FR}. We need more notation. Suppose that $V$ is a finite set of vertices and $\mathcal P^{(1)}$ is a partition of $V$ into sets $V_1,\dots, V_{t}$, which will be called \emph{clusters}. Given any $j\in [3]$, we denote by $\Cr_j=\Cr_j(\cP^{(1)})$ the set of all crossing $j$-subsets of $V$. For every set $A\subseteq [t]$ we write $\Cr_A$ for all the crossing subsets of $V$ that meet $V_i$ precisely when $i\in A$. Let $\cP_A$ be a partition of $\Cr_A$. We refer to the partition classes of $\cP_A$ as \emph{cells}. Let $\cP^{(2)}$ be the union of all $\cP_A$ with $|A|=2$ (so $\cP^{(2)}$ is a partition of $\Cr_2$). We call $\cP=\{\cP^{(1)}, \cP^{(2)}\}$ \emph{a family of partitions} on $V$.

Given $\cP=\{\cP^{(1)}, \cP^{(2)}\}$ and $K=v_i v_j v_k$ with $v_i\in V_i$, $v_j\in V_j$ and $v_k \in V_k$,  the \emph{polyad} $P(K)$ is the 3-partite 2-graph on $V_i\cup V_j\cup V_k$ with edge set $C(v_i v_j)\cup C(v_i v_k)\cup C(v_j v_k)$, where \eg, $C(v_i v_j)$ is the cell in $\cP_{i,j}$ that contains $v_i v_j$. We say that $P(K)$ is \emph{$(d_2,\delta)$-regular} if all $C(v_i v_j),C(v_i v_k), C(v_j v_k)$ are $(d_2,\delta)$-regular with respect to their underlying sets. We let $\hat {\cP}^{(2)}$ be the family of all $P(K)$ for $K\in \Cr_3$.

Now we are ready to state the regularity lemma for 3-graphs.

\begin{theorem}[R\"odl and Schacht \cite{RS}, Theorem 17]\label{thm:reg}
For all $\delta_3>0, t_0\in \mathbb N$ and all functions $r:\mathbb N \rightarrow \mathbb N$ and $\delta: \mathbb N\rightarrow (0,1]$, there are $d_2>0$ such that $1/d_2\in \mathbb N$ and integers $T, n_0$ such that the following holds for all $n\ge n_0$ that are divisible by $T!$. Let $\h$ be a 3-graph of order $n$. Then there exists a family of partitions $\cP=\{\cP^{(1)}, \cP^{(2)}\}$ of the vertex set $V$ of $\h$ such that
\begin{enumerate}
\item $\cP^{(1)}=\{V_1,\dots, V_t\}$ is a partition of $V$ into $t$ clusters of equal size, where $t_0\le t\le T$,
\item $\cP^{(2)}$ is a partition of $\Cr_2$ into at most $T$ cells,
\item for every $K\in \Cr_3$, $P(K)$ is $(d_2, \delta(T))$-regular,
\item $\sum |\K_3({P})|\le \delta_3 |V|^3$, where the summation is over all ${P}\in \hat {\cP}^{(2)}$ such that $\h$ is not $(d, \delta_3,r(T))$-regular with respect to ${P}$ for any $d> 0$.
\end{enumerate}
\end{theorem}

\subsection{The Reduced 3-graph and the Extension Lemma}

Given $t_0\in \mathbb N$ and $\delta_3 >0$, we choose functions $r:\mathbb N \rightarrow \mathbb N$ and $\delta: \mathbb N\rightarrow (0,1]$ such that the output of Theorem~\ref{thm:reg} satisfies the following hierarchy:
\begin{equation}
\label{eq:h1}
\frac{1}{n_0} \ll \left\{ \frac{1}{r}, \delta \right\} \ll \left\{ \delta_3, d_2, \frac{1}{T}\right\},
\end{equation}
where $r= r(T)$ and $\delta = \delta(T)$.
Let $\h$ be a 3-graph on $V$ of order $n\ge n_0$ such that $T!$ divides $n$. Suppose that $\cP=\{\cP^{(1)}, \cP^{(2)}\}$ satisfies Properties (1)--(4) given in Theorem~\ref{thm:reg}. For any $d>0$, the \emph{reduced 3-graph} $\R=\R(\h, \cP,d)$ is defined as the 3-graph whose vertices are clusters $V_1,\dots, V_{t}$ and three clusters $V_i, V_j, V_k$ form an edge of $\R$ if there is some polyad $P$ on $V_i\cup V_j\cup V_k$ such that $\h$ is $(d',\delta_3,r)$-regular with respect to ${P}$ for some $d'\ge d$.

\begin{fact}\label{fact:R}
Let $\R=\R(\h, \cP,d)$ be the reduced 3-graph defined above. If $V_i V_j V_k\in E(\R)$, then there exists a $(3,3)$-complex $\h^*$ on $V_i\cup V_j\cup V_k$ such that $\h^*_3$ is a subhypergraph of $\h$ and $\h^*$ is $(d', d_2, \delta_3, \delta, r)$-regular for some $d'\ge d$.
\end{fact}

\begin{proof}
Since $V_i V_j V_k\in E(\R)$, there exists a polyad $P$ on $V_i\cup V_j\cup V_k$ such that $\h$ is $(d', \delta_3,r)$-regular with respect to ${P}$ for some $d'\ge d$. Let $\h^*_2=P$ and $\h^*_3=E(\h) \cap \K_3(P)$. By Theorem \ref{thm:reg}, $\h^*$ is a $(d', d_2, \delta_3, \delta, r)$-regular $(3,3)$-complex.
\end{proof}

The following lemma says that the reduced 3-graph almost inherits the minimum degree condition from $\h$. Its proof is almost identical to the one of \cite[Lemma 4.3]{KMO}, which gives the corresponding result on co-degree. We thus omit the proof.

\begin{lemma}\label{lem:deg}
In addition to \eqref{eq:h1}, suppose that
\[
\delta_3,{1}/{t_0} \ll d\ll \theta\ll \mu<1.
\]
Let $\h$ be a 3-graph of order $n\ge n_0$ such that $T!$ divides $n$ and $\delta_1(\h)\ge (\mu+\theta) \binom n2$. Then in the reduced 3-graph $\R=\R(\h, \cP,d)$, all but at most $\theta t$ vertices $v\in V(\R)$ satisfy $\deg_{\R}(v)\ge \mu \binom{t}2$.
\end{lemma}

Suppose that $\h$ is a $(3,3)$-complex with vertex classes $V_1, V_2, V_3$, and $\G$ is a $(3,3)$-complex with vertex classes $X_1, X_2, X_3$. A subcomplex $\h'$ of $\h$ is called a \emph{partition-respecting copy of} $\G$ if $\h'$ is isomorphic to $\G$ and for each $i\in [3]$ the vertices corresponding to those in $X_i$ lie within $V_i$.
We write $|\G|_{\h}$ for the number of (labeled) partition-respecting copies of $\G$ in $\h$.

Roughly speaking, the Extension Lemma \cite[Lemma 5]{CFKO}  says that if $\G'$ is an induced subcomplex of $\G$, and $\h$ is suitably regular, then almost all copies of $\G'$ in $\h$ can be extended to a large number of copies of $\G$ in $\h$. Below we only state it for $(3,3)$-complexes.  
\begin{lemma}[Extension Lemma \cite{CFKO}]
\label{lem:ext}
Let $r,b,b',m_0$ be positive integers, where $b'<b$, and let $c, \theta, d_2, d_3, \delta, \delta_3$ be positive constants such that $1/d_2\in \mathbb N$ and
\[
1/m_0\ll \{1/r, \delta\} \ll c \ll \min\{\delta_3, d_2\}\le \delta_3 \ll \theta, d_3, 1/b.
\]
Then the following holds for all integers $m\ge m_0$. Suppose that $\G$ is a $(3,3)$-complex on $b$ vertices with vertex classes $X_1,X_2, X_3$ and let $\G'$ be an induced subcomplex of $\G$ on $b'$ vertices. Suppose also that $\h^*$ is a $(d_3,d_2, \delta_3, \delta, r)$-regular $(3, 3)$-complex with vertex classes $V_1, V_2, V_3$, all of size $m$ and $e(\h^*)>0$. Then all but at most $\theta |\G'|_{\h^*}$ labeled partition-respecting copies of $\G'$ in $\h^*$ are extendible to at least $c m^{b-b'}$ labeled partition-respecting copies of $\G$ in $\h^*$.
\end{lemma}

\section{Proof of Lemma \ref{lemA}}

In this section we prove Lemma \ref{lemA} by using the lemmas introduced in Section~2. We remark that the constant $\frac14$ in Lemma \ref{lemA}  is best possible because if $\h$ consists of two disjoint cliques of order $n/2$ each, then $\delta_1(\h)$ is about $\frac14 \binom{n}{2}$ and any 4-vertex set that intersects both cliques can not be absorbed.

For $\a>0$, $i\in \mathbb N$ and two vertices $u,v\in V$, we say that $u$ is $(\a ,i)$-reachable to $v$ if and only if there are at least $\a n^{4i-1}$ $(4i-1)$-sets $W$ such that both $\h[u\cup W]$ and $\h[v\cup W]$ contain $\C$-factors. In this case, we call $W$ a reachable set for $u$ and $v$. Similar definitions for absorbing method can be found in \cite{LM1, LM2}. Suppose that
\[
1/n_0\ll \left\{1/r, \delta \right\} \ll c \ll \min \left\{\delta_3, {1}/{T}, d_2\right\}\le \delta_3, 1/t_0\ll d\ll \theta\le 10^{-4},
\]
and $n_0 \ge 4T!/\theta$.
Let $\h$ be a 3-graph on $n\ge n_0+ T!$ vertices with $\delta_1(\h)\ge (\frac14+\theta)\binom n2$. We will prove that almost all pairs of vertices of $\h$ are $(\beta_0,2)$-reachable to each other, where $\beta_0=c^2 /(5T^{7})$.

\begin{claim}\label{clm:bad}
There are at most $4\theta n^2$ pairs $u, v\in \binom V2$ such that $u$ is not $(\beta_0 ,2)$-reachable to $v$.
\end{claim}

\begin{proof}
Let $n'\in \mathbb{N}$ such that $n-n'<T!$ and $T!$ divides $n'$. Then $n'\ge n_0\ge 4T!/\theta$. As $\theta \le 10^{-4}$, we have $n'\ge \frac{40000}{40001} n$.

Let $\h'$ be an induced subhypergraph of $\h$ on any $n'$ vertices. Since $n\ge 4T!/\theta$, we have
\[
\delta_1(\h')\ge \left(\frac14+\theta \right)\binom n2 - T! (n-1) \ge \left(\frac14+\frac{\theta}2 \right)\binom {n'}2.
\]
We apply Theorem~\ref{thm:reg} to $\h'$, and let $\cP$ be the the family of partitions, with clusters $V_1,\dots,V_t$. Let $m=n'/t$ be the size of each cluster. Define the reduced 3-graph $\R=\R(\h', \cP, d)$ on these clusters as in Section 2.3.

Let $I$ be the set of $i\in [t]$ such that $\deg_{\R}(V_i)<(\frac14+\frac{\theta}4)\binom {t}2$ and let $V_I=\bigcup_{i\in I}V_i$. By Lemma \ref{lem:deg}, we have $|I| \le \theta t/4$ and thus $|V_I|\le (\theta t/4)\cdot m=\theta n'/4$.
Let $N(i)$ be the set of vertices $V_j\in V(\R)\setminus \{V_i\}$ such that $\{V_i, V_j \} \subseteq e$ for some $e\in \R$. For any $i\in [t]\setminus I$,
\[
\left(\frac14+\frac{\theta}4 \right)\binom {t}2\le \deg_{\R}(V_i)\le \binom{|N(i)|}2
\]
implies that $|N(i)|\ge (\frac12+\frac{\theta}8)t$. Thus $|N(i)\cap N(j)|\ge \frac{\theta}4 t$ for any $i, j\in [t]\setminus I$.

Fix two not necessarily distinct $i, j\notin I$ and $V_k\in N(i)\cap N(j)$. We pick $V_{i'}$ and $V_{j'}$ such that $V_i V_k V_{i'}, V_j V_k V_{j'}\in \R$. Note that it is possible to have $i' = j'$ or $i' = j$ or $j' = i$.
Let $\h^*$ be the $(d_3,d_2, \delta_3, \delta, r)$-regular $(3, 3)$-complex with vertex classes $V_i, V_k, V_{i'}$ provided by Fact \ref{fact:R}, where $d_3\ge d$.

Let $\G$ be the (3,3)-complex on $X_1=\{x, u\}$, $X_2=\{y,v\}$, $X_3=\{w\}$ such that $\G_3=\{xvw, uyw, uvw\}$ and $\G_2$ is the family of all 2-subsets of the members of $\G_3$. Note that in $\G_3$ both $\{x,u,v,w\}$ and $\{y,u,v,w\}$ span copies of $\C$.
Let $\G'$ be the induced subcomplex of $\G$ on $\{u, v\}$.
Since $\h^*_3$, the highest level of the complex $\h^*$, is not empty, by Lemma \ref{lem:ext}, all but at most $\theta m^2$ ordered pairs $(v_i, v_k)\in V_i\times V_k$ are extendible to at least $c m^{3}$ labeled copies of $\G$ in $\h^*$, which implies that $v_i$ is $(cm^3 n^{-3},1)$-reachable to $v_k$.
By averaging, all but at most $3\theta m$ vertices $v_i\in V_i$ are $(cm^3 n^{-3},1)$-reachable to at least $\frac 23m$ vertices of $V_k$. We apply the same argument on $V_j, V_k, V_{j'}$ and obtain that for all but at most $3\theta m$ vertices $v_j\in V_j$, $v_j$ is $(cm^3 n^{-3},1)$-reachable to at least $\frac 23m$ vertices of $V_k$. Thus for those $v_i$ and $v_j$, there are $\frac 13m$ vertices $v_k\in V_k$ such that both $v_i$ and $v_j$ are $(cm^3 n^{-3},1)$-reachable to $v_k$.
Fix such $v_i, v_j, v_k$. There are at least $c m^3-m^2$ reachable 3-sets for $v_i$ and $v_k$ from $(V_i, V_{i'}, V_k)$ avoiding $v_j$.\footnote{Recall that it is possible to have $v_j\in V_{i}$ or $v_j\in V_{i'}$ (when $j=i$ or $j=i'$).} Fix one such 3-set, the number of 3-sets from $(V_j, V_{j'}, V_k)$ intersecting its three vertices is at most $3m^2$. So the number of reachable 7-sets for $v_i, v_j$ is at least
\[
\frac{m}3\cdot (c m^3- m^2)\cdot (c m^3-3m^2) >\frac{c^2}4 m^7 \ge \frac{c^2}4 \left(\frac{n'}{T} \right)^7 > \frac{c^2}5 \frac{n^7}{T^{7}} =\beta_0 n^7,
\]
which means that $v_i$ is $(\beta_0,2)$-reachable to $v_j$, where the last inequality holds because $(\frac{n'}{n})^7 \ge (\frac{40000}{40001})^7 > \frac{4}{5}$. Note that this is true for all but at most $2\cdot 3\theta m\cdot m=6\theta m^2$ pairs of vertices in $(V_i, V_j)$. 
Since there are at most $\binom{t}{2} + t$ choices for $V_i$ and $V_j$, $|V_I| \le \theta n' /4$, and $T! \le \theta n'/4$, there are at most
\begin{align*}
& 6\theta m^2  \left(\binom {t}2 + t \right)  +(|V_I|+T!)(n-1) \le 3\theta m^2 (t^2 + t) + \frac{\theta}2 n' (n-1) \le 4 \theta n^2
\end{align*}
pairs $u, v$ in $V(\h)$ such that $u$ is not $(\beta_0 ,2)$-reachable to $v$.
\end{proof}

\medskip

\begin{proof}[Proof of Lemma \ref{lemA}]
Let $\beta=\beta_0^{10}$. Let $V'$ be the set of vertices $v\in V$ such that at least $\frac n{64}$ vertices are not $(\beta_0,2)$-reachable to $v$. By Claim \ref{clm:bad}, $|V'|\le 512{\theta} n$.

There are two steps in our proof. In the first step, we build an absorbing family $\F'$ such that for any small portion of vertices in $V(\h)\setminus V'$, we can absorb them using members of $\F'$. In the second step, we put the vertices in $V'$ not covered by any member of $\F'$ into a set $\A$ of copies of $\C$. Thus, the union of $\F'$ and $\A$ gives the desired absorbing set.

We say that a set $A$ absorbs another set $B$ if $A\cap B=\emptyset$ and both $\h[A]$ and $\h[A\cup B]$ contains $\C$-factors. Fix any 4-set $S=\{v_1,v_2,v_3,v_4\}\in V\setminus V'$, we will show that there are many 24-sets absorbing $S$. First, we find vertices $u_2, u_3, u_4$ such that
\begin{itemize}
\item $v_1u_2u_3u_4$ spans a copy of $\C$,
\item $u_i$ is $(\beta_0, 2)$-reachable to $v_i$, for $i=2,3,4$.
\end{itemize}
For the first condition, consider the link graph $\h_{v_1}$ of $v_1$, which contains at least $(\frac14 +\theta) \binom n2$ edges. By convexity, the number of paths of length two in $\h_{v_1}$ is
\begin{align*}
\sum_{x\in V\setminus \{v_1\}}\binom{\deg_{\h_{v_1}}(x)}2&\ge (n-1)\binom{\frac1{n-1}\sum_{x\in V\setminus \{v_1\}}\deg_{\h_{v_1}}(x)}2\\
&\ge (n-1) \binom{(\frac14 +\theta) n}2 > \frac 1{32}n^3,
\end{align*}
where the last inequality holds because $\theta n\gg 1$.
Since $v_1u_2u_3u_4$ spans a copy of $\C$ if $u_2u_3u_4$ is a path of length two in $\h_{v_1}$, then there are at least $\frac 1{32} n^3$ choices for such $u_2u_3u_4$. Moreover, the number of triples violating the second condition is at most $3\cdot\frac n{64}\cdot \binom n2<\frac 3{128}n^3$. Thus, there are at least $\frac 1{128} n^3$ such $u_2u_3u_4$ satisfying both of the conditions.

Second, we find reachable 7-sets $C_i$ for $u_i$ and $v_i$, for $i=2,3,4$, which is guaranteed by the second condition above. Since in each step we need to avoid at most 21 previously selected vertices, there are at least $\frac{\beta_0}2 n^7$ choices for each $C_i$. In total, we get $\frac 1{128} n^3\cdot (\frac{\beta_0}2 n^7)^3>\beta_0^4 n^{24}$ 24-sets $F=C_1\cup C_2\cup C_3\cup\{u_2, u_3, u_4\}$ (because $\beta_0 < c^2 < 10^{-8}$).
It is easy to see that $F$ absorbs $S$. Indeed, $\h[F]$ has a $\C$-factor since $C_i\cup\{u_i\}$ spans two copies of $\C$ for $i=2,3,4$. In addition, $\h[F\cup S]$ has a $\C$-factor since $v_1u_2u_3u_4$ spans a copy of $\C$ and $C_i\cup\{v_i\}$ spans two copies of $\C$ for $i=2,3,4$.

Now we choose a family $\mathcal F\subset \binom{V}{24}$ of 24-sets by selecting each $24$-set randomly and independently with probability $p=\beta_0^5 n^{-23}$. Then $|\F|$ follows the binomial distribution $B(\binom{n}{24}, p)$ with expectation $\mathbb{E}(|\F|)=p\binom{n}{24}$. Furthermore, for every 4-set $S$, let $f(S)$ denote the number of members of $\F$ that absorb $S$. Then $f(S)$ follows the binomial distribution $B(N, p)$ with $N\ge \beta_0^4 n^{24}$ by previous calculation. Hence $\mathbb{E}(f(S))\ge p\beta_0^4 n^{24}$. Finally, since there are at most $\binom n{24}\cdot 24\cdot \binom n{23}<\frac12n^{47}$ pairs of intersecting 24-sets, the expected number of the intersecting pairs of 24-sets in $\F$ is at most $p^2\cdot \frac12n^{47}=\beta_0^{10}n/2$. 
 
Applying Chernoff's bound on the first two properties and Markov's bound on the last one, we know that, with positive probability, $\mathcal F$ satisfies the following properties:
\begin{itemize}
\item $|\F|\le 2p \binom n{24}<\beta_0^5 n$,
\item for any 4-set $S$, $f(S)\ge \frac p2\cdot \beta_0^4 n^{24}=\beta_0^9 n/2$,
\item the number of intersecting pairs of elements in $\F$ is at most $\beta_0^{10}n$.
\end{itemize}
Thus, by deleting one member from each intersecting pair and the non-absorbing members from $\F$, we obtain a family $\F'$ consisting of at most $\beta_0^{5}n$ 24-sets and for each 4-set $S$, at least $\beta_0^9 n/2-\beta_0^{10} n>\beta_0^{10} n=\beta n$ members in $\F'$ absorb $S$.

At last, we will greedily build $\A$, a collection of copies of $\C$ to cover the vertices in $V'$ not already covered by any member of $\F'$. Indeed, assume that we have built $a<|V'|\le 512\theta n$ copies of $\C$. Together with the vertices in $\F'$, there are at most $4a+24\beta_0^5 n<2049\theta n$ vertices already selected. Then at most $2049{\theta} n^2$ pairs of vertices intersect these vertices. So for any remaining vertex $v\in V'$, there are at least
\[
\deg(v)-1025{\theta} n^2\ge \left(\frac14 +\theta\right) \binom n2 - 2049{\theta} n^2>n/2
\]
edges containing $v$ and not intersecting the existing vertices, where the last inequality follows from $\theta\le 10^{-4}$. So there is a path of length two in the link graph of $v$ not intersecting the existing vertices, which gives a copy of $\C$ containing $v$.

Combining the vertices covered by $\A$ and $\F'$ together, we get the desired absorbing set $W$ satisfying $|W|\le 4\cdot 512{\theta} n+24\beta_0^5 n<2049\theta n$.
\end{proof}

\section{Proof of Theorem \ref{lemE}}

In this section we prove Theorem \ref{lemE}. Our proof is similar to the one of \cite[Theorem 1.4]{CDN}.
First let us start with some notation. Fix a 3-graph $\h= (V, E)$.
Recall that the link graph of a vertex $v\in V$ is a 2-graph on $V\setminus \{v\}$. Then for a set $\mathcal{E}$ of pairs in $\binom{V}{2}$ (which can be viewed as a 2-graph),
let $\deg_{\h}(v, \mathcal{E})=|N_{\h}(v)\cap \mathcal{E}|$. When $\mathcal{E}=\binom{X}{2}$ for some $X\subseteq V$, we write $\deg_{\h}(v, \binom{X}{2})$ as $\deg_{\h}(v, X)$ for short.
Let $\overline\deg_{\h}(v, \mathcal{E})=|\mathcal{E}\cap \binom{V\setminus \{v\}}{2}| - \deg_{\h}(v, \mathcal{E})$.
Given not necessarily disjoint subsets $X, Y, Z$ of $V$, define
\begin{align*} 
e_{\h}(X Y Z) &= \{xyz \in E(\h): x\in X, y\in Y, z\in Z\} \\
\overline{e}_{\h}(X Y Z) &= \left\{xyz \in \binom{V}{3}\setminus E(\h): x\in X, y\in Y, z\in Z \right\}.
\end{align*}
We often omit the subscript $\h$ if it is clear from the context.

The following fact is the only place where we need the \emph{exact} degree condition \eqref{eqdeg}.

\begin{fact}\label{fact:ext}
Let $\h$ be a 3-graph on $n$ vertices with $n\in 4\mathbb N$ satisfying \eqref{eqdeg}. If $S\subseteq V(\h)$ spans no copy of $\C$, then $|S|\le \frac34n$.
\end{fact}

\begin{proof}
Assume to the contrary, that $S\subseteq V(\h)$ spans no copy of $\C$ and is of size at least $\frac34n+1$. Take $S_0\subseteq S$ with size exactly $\frac34n+1$. Then 
for any $v\in S_0$, $\deg(v, S_0)\le \frac{|S_0|-1}2=\frac38n$. We split into two cases.
\begin{case}
$n\in 8\mathbb N$.
\end{case}
In this case, for any $v\in S_0$, since $\deg(v, S_0)\le \frac38n$, we have that
\[
\deg(v)=\deg(v, S_0)+\deg\left(v, \binom V2\setminus \binom {S_0}2\right)\le \frac38n+\binom{n-1}2-\binom{\frac34n}2<\delta_1(\h),
\]
contradicting \eqref{eqdeg}.

\begin{case}
$n\in 4\mathbb N\setminus 8\mathbb N$.
\end{case}
In this case, for any $v\in S_0$, $\deg(v, S_0)\le \frac38n$ implies that $\deg(v, S_0) \le \frac38n-\frac12$ because $n\in 4\mathbb N\setminus 8\mathbb N$. So we have
\[
3e(S_0)=\sum_{v\in S_0}\deg(v,S_0)\le \left(\frac38n-\frac12\right) \left( \frac34n+1 \right)= \frac{3n-4}8\cdot \frac{3n+4}4.
\]
However, neither $\frac{3n-4}8$ or $\frac{3n+4}4$ is a multiple of 3. Thus $\sum_{v\in S_0}\deg(v,S_0)<\frac{3n-4}8\cdot \frac{3n+4}4$, which implies that there exists $v_0\in S_0$ such that $\deg(v_0, S_0)<\frac38n-\frac12$. Consequently,
\[
\deg(v_0)=\deg(v_0, S_0)+\deg\left(v_0, \binom V2\setminus \binom {S_0}2\right)< \frac38n-\frac12+\binom{n-1}2-\binom{\frac34n}2\le\delta_1(\h),
\]
contradicting \eqref{eqdeg}.
\end{proof}

\begin{proof}[Proof of Theorem \ref{lemE}]
Take ${\e}=10^{-18}$ and let $n$ be sufficiently large. We write $\a ={\e}^{1/3} = 10^{-6}$. Let $\h=(V, E)$ be a 3-graph of order $n$ satisfying \eqref{eqdeg} which is ${\e}$-extremal, namely, there exists a set $S\subseteq V(\h)$ such that $|S| \ge (1-\e)\frac{3n}4$ and $\h[S]$ is $\C$-free.

Let $C\subseteq V$ be a maximum set for which $\h[C]$ is $\C$-free. Define
\begin{equation}\label{eq:A}
A=\left\{x\in V\setminus C: \deg(x, C)\ge (1-\a) \binom{|C|}2\right\},
\end{equation}
and $B=V\setminus (A\cup C)$. We first claim the following bounds of $|A|, |B|, |C|$.
\begin{claim}\label{clm:size}
$|A|>\frac n4(1-4\a^2), |B|<\a^2 n$ and $\frac {3n}4(1-{\e})\le |C|\le \frac {3n}4$.
\end{claim}

\begin{proof}
The estimate on $|C|$ follows from our hypothesis and Fact \ref{fact:ext}. We now estimate $|B|$. For any $v\in C$, we have $\deg(v, C)\le \frac{|C|-1}2$, which gives $\overline{\deg}(v, C)\ge \binom{|C|-1}2 - \frac{|C|-1}2$. By \eqref{eqdeg}, $\overline{\deg}(v) \le \binom{\frac34n}2 - \frac38n+\frac12$. Thus
\begin{align*}
\overline{\deg}\left( v, \binom V2\setminus \binom C2 \right)&\le \binom{\frac34n}2 - \frac38n+\frac12 - \binom{|C|-1}2 + \frac{|C|-1}2  \\
&\le \binom{\frac34n}2 - \binom{|C|-1}2 \quad \text{because }|C|\le \frac34n \\
&= \left(\frac34n-|C|+1\right) \cdot \frac12 \left(\frac34n+|C|-2\right). 
\end{align*}
The estimate on $|C|$ gives $\frac34n\le \frac{|C|}{1-{\e}}<(1+2{\e})(|C|-1)$. 
Hence
\begin{align}
\overline{\deg}\left( v, \binom V2\setminus \binom C2 \right)
&< \left(\frac34n-|C|+1\right) \cdot \frac12 \biggl((1+2{\e})(|C|-1)+|C|-1\biggr) \nonumber \\
&= \left(\frac34n-|C|+1\right) \cdot (1+{\e})(|C|-1) \label{eq:degC}\\
&\le \left(\frac34{\e} n+1\right) \cdot (1+{\e})(|C|-1) <{\e} n\cdot(|C|-1). \label{eq:degCC}
\end{align}
Consequently $\overline{e}(CC(A\cup B))< \frac12|C|\cdot {\e} n\cdot (|C|-1)={\e} n\cdot \binom{|C|}2$. Together with the definition of $A$ and $B$, we have
\[
(|A\cup B|-{\e} n)\binom{|C|}2 < e(CC(A\cup B))\le (1-\a)\binom{|C|}2 |B| + \binom{|C|}2 |A|,
\]
so that $|A\cup B|-{\e} n<|A|+|B|-\a |B|$. Since $A$ and $B$ are disjoint, we get that $|B|<\a^2 n$. Finally, $|A|=n-|B|-|C|> n-\a^2 n - \frac34n=\frac n4(1-4\a^2)$.
\end{proof}

In the rest of the section, we will build four vertex-disjoint $\C$-tilings $\Q, \R, \ss, \T$ whose union is a perfect $\C$-tiling of $\h$. In particular, when $|A|=n/4$, $B=\emptyset$ and $|C|=3n/4$, we have $\Q=\R=\ss=\emptyset$ and the perfect $\C$-tiling $\T$ of $\h$ will be provided by Lemma \ref{lem3}. The purpose of $\C$-tilings $\Q, \R, \ss$ is covering the vertices of $B$ and adjusting the sizes of $A$ and $C$ such that we can apply Lemma \ref{lem3} after $\Q, \R, \ss$ are removed.

\medskip
\noindent {\bf The $\C$-tiling $\Q$.} Let $\Q$ be a largest $\C$-tiling in $\h$ on $B\cup C$ and $q=|\Q|$. 
We claim that $|B|/4\le q\le |B|$. Since $C$ contains no copy of $\C$, every element of $\Q$ contains at least one vertex of $B$ and consequently $q\le |B|$. On the other hand, suppose that $q<|B|/4$, then $(B\cup C)\setminus V(\Q)$ spans no copy of $\C$ and has order
\[
|B|+|C|-4q>|B|+|C|-|B|=|C|.
\]
which contradicts the assumption that $C$ is a maximum $\C$-free subset of $V(\h)$.

\begin{claim}\label{clm:q}
$q + |A| \ge \frac n4$.
\end{claim}

\begin{proof}
Let $l = \frac n4-|A|$. There is nothing to show if $l\le 0$. 
If $l=1$, we have $|B\cup C|=\frac34n+1$, and thus Fact \ref{fact:ext} implies that $\h[B\cup C]$ contains a copy of $\C$. Thus $q\ge 1=l$ and we are done. We thus assume $l \ge 2$ and $l>q\ge |B|/4$, which implies that $|B|\le 4(l-1)$. 
In this case $|B|\ge 2$ because $|C|\le \frac34n$. 

For any $v\in C$, by \eqref{eq:degC}, we have $\overline{\deg}(v, BC)<\left(\frac34n-|C|+1\right) \cdot (1+{\e})(|C|-1)$. By definition, $\frac34n-|C|=|A|+|B|-\frac n4=|B|-l$. So we get
\[
\overline{e}(BCC)<\frac12 |C|\left(\frac34n-|C|+1\right) \cdot (1+{\e})(|C|-1)=(1+{\e})(|B|-l+1)\binom{|C|}2.
\]
Together with $|B|\le 4(l-1)$, this implies
\begin{align}
e(BCC)&>(|B|-(1+{\e})(|B|-l+1))\binom{|C|}2 \nonumber \\
&=((1+{\e})(l-1)-{\e}|B|)\binom{|C|}2 \nonumber \\
&\ge ((1+{\e})(l-1) - 4{\e} (l-1))\binom{|C|}2=(1 - 3{\e})(l-1)\binom{|C|}2. \label{eq:upperbcc}
\end{align}

On the other hand, we want to bound $e(BCC)$ from above and then derive a contradiction. Assume that $\Q'$ is the maximum $\C$-tiling of size $q'$ such that each element of $\Q'$ contains exactly one vertex in $B$ and three vertices in $C$. 
Note that $q'\ge 1$ because $C$ is a maximum $\C$-free set and $B\ne \emptyset$. 
Write $B_{\Q'}$ for the set of vertices of $B$ covered by $\Q'$ and $C_{\Q'}$ for the set of vertices of $C$ covered by $\Q'$. Clearly, $|B_{\Q'}|=q'$, $|C_{\Q'}|=3q'$ and $q'\le q\le l-1$. For any vertex $v\in B\setminus B_{\Q'}$, $\deg(v, C)\le 3q' (|C|-1)+\frac12 |C|<4q' |C|$. Together with the definition of $B$ and Claim \ref{clm:size}, we get
\begin{align}
e(BCC)&=e(B_{\Q'}CC)+e((B\setminus B_{\Q'})CC)\nonumber \\
&\le q' (1-\a)\binom {|C|}2 + |B|\cdot 4q'|C|< q' (1-\a)\binom {|C|}2 +4\a^2 nq' |C|. \label{eq:lowerbcc}
\end{align}
Putting \eqref{eq:upperbcc} and \eqref{eq:lowerbcc} together and using $q'\le l-1$ and $|C|>n/2$, we get
\[
1-3\a^3=1-3{\e}<1-\a + \frac{8\a^2 n}{|C|-1}<1-\a + 16\a^2<1-\frac{\a}2,
\]
which is a contradiction since $\a=10^{-6}$.
\end{proof}

Let $B_1$ and $C_1$ be the vertices in $B$ and $C$ not covered by $\Q$, respectively. By Claim~\ref{clm:size},
\begin{equation}\label{eq:C_1}
|C_1|\ge |C|-3q\ge |C|-3|B|> \frac34n(1-{\e})-3\a^2 n>\frac34n-4\a^2n +1.
\end{equation}

\medskip
\noindent {\bf The $\C$-tiling $\R$.} Next we will build our $\C$-tiling $\R$ which covers $B_1$ such that every element in $\R$ contains one vertex from $A$, one vertex from $B_1$ and two vertices from $C_1$. Since $\Q$ is a maximum $\C$-tiling on $B\cup C$, for every vertex $v\in B_1$, we have that $\deg(v, C_1)\le \frac{|C_1|}2$. Together with \eqref{eq:C_1}, this implies that
\[
\overline{\deg}(v,C_1)\ge \binom{|C_1|}2-\frac{|C_1|}2=\frac{|C_1|(|C_1|-2)}2 >\frac{(\frac34n-4\a^2n)^2-1}2.
\]
Together with \eqref{eqdeg}, we get that for every $v\in B_1$,
\begin{align*}
\overline{\deg}(v, AC_1)&<\binom{\frac34n}2-\frac38n+\frac12 - \frac{(\frac34n-4\a^2n)^2-1}2 \\
&=\frac12\left(\frac32n-4\a^2 n \right) 4\a^2 n -\frac34n+1 < 3\a^2n^2.
\end{align*}
By Claim \ref{clm:size} and \eqref{eq:C_1}, we have that $|A||C_1|> (1-4\a^2)\frac n4 \cdot (\frac34 - 4\a^2)n>\frac3{17}n^2$. Thus, $\overline{\deg}(v, AC_1)<3\a^2n^2<17\a^2|A| |C_1|$, equivalently, $\deg(v, AC_1)> (1-17\a^2)|A||C_1|$. For every $v\in B_1$, we greedily pick a copy of $\C$ containing $v$ 
by picking a path of length two with center in $A$ and two ends in $C_1$ from the link graph of $v$.
Suppose we have found $i<|B_1|$ copies of $\C$, then for any remaining vertex $v\in B_1$, by Claim \ref{clm:size}, the number of pairs not intersecting the existing vertices is at least
\[
\deg(v, AC_1) - 3i\cdot (|A|+|C_1|)> (1-17\a^2)|A||C_1|-3|B_1|\cdot 2|C_1|>|A|,
\]
which guarantees a path of length two centered at $A$, so a copy of $\C$ containing $v$.

\medskip

Now all vertices of $B$ are covered by $\Q$ or $\R$. Let $A_2$ denote the set of vertices of $A$ not covered by $\Q$ or $\R$ and define $C_2$ similarly. 
By the definition of $\Q$ and $\R$, we have $|A_2|=|A| - |B_1|$ and $|C_2| =|B|+|C|-4q-3|B_1|$. Define $s=\frac14(3|A_2|-|C_2|)$. Then 
\begin{align*}
s =\frac14( 3|A| - 3|B_1| - |B| - |C| + 4q + 3|B_1|) = \frac14 ( 4|A| - n + 4q) = q+|A|-\frac n4.
\end{align*}
Thus $s\in \mathbb Z$, and $s\ge 0$ by Claim~\ref{clm:q}. Since $q\le |B|$, by Claim \ref{clm:size},
\begin{equation}\label{eq:ss}
s= q+|A|-\frac n4\le |B|+|A|-\frac n4=\frac34n-|C|\le \frac34{\e} n.
\end{equation}
The definition of $\Q$ and $\R$ also implies that $|C\setminus C_2|\le 3|B|$ and 
\begin{equation}\label{eq:C_2}
|C_2|\ge |C|-3|B|> |C|-3\cdot 2\a^2 |C|=(1-6\a^2)|C|,
\end{equation}
where the second inequality follows from $|B|< \a^2 n<2\a^2 |C|$.

\medskip
\noindent {\bf The $\C$-tiling $\ss$.} Next we will build our $\C$-tiling $\ss$ of size $s$ such that every element of $\ss$ contains two vertices in $A_2$ and two vertices in $C_2$. Note that for any vertex $v\in A_2$, by \eqref{eq:A} and \eqref{eq:C_2},
\[
\overline{\deg}(v, C_2)\le \a \binom {|C|}2\le \a \binom {\frac{1}{1-6\a^2}|C_2|}2< 2\a\binom{|C_2|}2.
\]
Suppose that we have found $i<s$ copies of $\C$ of the desired type. We next select two vertices $a_1, a_2$ in $A_2$ and note that they have at least $(1-4\a)\binom{|C_2|}2$ common neighbors in $C_2$. By \eqref{eq:ss} and\eqref{eq:C_2}, 
\[
(1-4\a)\binom{|C_2|}2 - 2s|C_2|\ge (1-4\a)\binom{|C_2|}2 - \frac32{\e} n|C_2|\ge (1-5\a)\binom{|C_2|}2>0.
\]
So we can pick a common neighbor $c_1c_2$ of $a_1$ and $a_2$ from unused vertices of $C_2$ such that $\{a_1, a_2, c_1, c_2\}$ spans a copy of $\C$.

\medskip

Let $A_3$ be the set of vertices of $A$ not covered by $\Q, \R, \ss$ and define $C_3$ similarly. Then $|A_3|=|A_2|-2s=\frac12(|C_2|-|A_2|)$ and $|C_3|=|C_2|-2s=\frac32(|C_2|-|A_2|)$, so $|C_3|=3|A_3|$. Furthermore, by \eqref{eq:ss} and \eqref{eq:C_2}, we have 
\[
|C_3|= |C_2|-2s\ge (1-6\a^2)|C|-\frac32{\e} n> (1-6\a^2)|C|-3{\e} |C|> (1-7\a^2)|C|.
\]
Hence, for every vertex $v\in A_3$,
\[
\overline{\deg}(v, C_3)\le \a \binom {|C|}2\le \a \binom {\frac{1}{1-7\a^2}|C_3|}2< 2\a \binom{|C_3|}2.
\]
 Since $|C_3|\ge (1-7\a^2)|C|\ge (1-7\a^2)(1-{\e})\frac34n $, by \eqref{eq:degCC}, we know that for any vertex $v\in C_3$,
\begin{align*}
\overline{\deg}(v, A_3C_3)&<{\e} n\cdot(|C|-1) <2{\e} |C_3|^2=6\e |A_3| |C_3|.
\end{align*}

\medskip
\noindent {\bf The $\C$-tiling $\T$.} Finally we use the following lemma to find a $\C$-tiling $\T$ covering $A_3$ and $C_3$ such that every element of $\T$ contains one vertex in $A_3$ and three vertices in $C_3$. Note that in \cite{CDN}, this was done by applying a general theorem of Pikhurko \cite[Theorem 3]{Pik} (but impossible here because we do not have the co-degree condition).
\begin{lemma}\label{lem3}
Suppose that $0<\rho\le 2\cdot 10^{-6}$ and $n_{\ref{lem3}}$ is sufficiently large. Let $\h$ be a 3-graph on $n\ge n_{\ref{lem3}}$ vertices with $V(\h)=X\dot\cup Z$ such that $|Z|= 3|X|$. Further, assume that for every vertex $v\in X$, $\overline \deg(v, Z)\le \rho\binom{|Z|}2$ and for every vertex $v\in Z$, $\overline \deg (v, XZ)\le \rho |X| |Z|$. Then $\h$ contains a perfect $\C$-tiling.
\end{lemma}

Applying Lemma \ref{lem3} with $X=A_3$, $Z=C_3$, $\rho=2\a$ finishes the proof of Theorem \ref{lemE}.
\end{proof}

\begin{proof}[Proof of Lemma \ref{lem3}]
Let us outline the proof first. Let $X=\{x_1,\dots, x_{|X|}\}$. Our goal is to partition the vertices of $Z$ into $|X|$ triples $\{Q_1,\dots,Q_{|X|}\}$ such that for every $i=1,\dots, |X|$, $\{x_i\}\cup Q_i$ spans a copy of $\C$ -- in this case we say $Q_i$ and $x_i$ are \emph{suitable} for each other. From our assumptions, every $x\in X$ is suitable for most triples of $\C$, and most triples of $\C$ are suitable for most vertices of $X$. However, once we partition $\C$ into a particular set of triples $\{Q_1,\dots,Q_{|X|}\}$, we can not guarantee that every vertex in $X$ is suitable for most $Q_i$'s. To handle this difficulty, we use the absorbing method -- first find a small number of triples that can absorb any small(er) amount of vertices of $X$ and then extend it to a partition $\{Q_1,\dots,Q_{|X|}\}$ covering $Z$, and finally apply the greedy algorithm and the Marriage Theorem to find a perfect matching between $X$ and $\{Q_1,\dots,Q_{|X|}\}$.
Note that a similar approach was outlined in \cite{Khan2} to prove the extremal case.

We now start our proof. Let $G$ be the graph of all pairs $uv$ in $Z$ such that $\deg(uv, X)\ge (1-\sqrt{\rho})|X|$. We claim that
\begin{equation}\label{eq:dG}
\delta(G)\ge |Z| - \sqrt{\rho} |Z| - 1.
\end{equation}
Otherwise, some vertex $v\in Z$ satisfies $\deg_G (v)<  |Z| - \sqrt{\rho} |Z| - 1$, equivalently, 
$\overline \deg_G (v)> \sqrt{\rho}|Z|$. As each $u\notin N_G(v)$ satisfies $\overline{\deg}(uv, X)> \sqrt{\rho}|X|$, we have
\[
\overline \deg(v, XZ)> \sqrt{\rho}|Z| \cdot \sqrt{\rho} |X| = {\rho}|Z||X|,
\]
contradicting our assumption.
We call a triple $z_1z_2z_3$ in $Z$ \emph{good} if $G[z_1z_2z_3]$ contains at least two edges, otherwise \emph{bad}. Since a bad triple contains at least two non-edges of $G$, by \eqref{eq:dG}, the number of bad triples in $Z$ is at most
\begin{equation*}\label{eq:good}
\sum_{x\in Z}\binom{\overline\deg_{G}(x)}2\le |Z|\binom{\sqrt{\rho} |Z|}2 \le 3 {\rho}\binom{|Z|}3.
\end{equation*}
If $z_1 z_2 z_3$ is good, then by the definition of $G$, it is suitable for at least $(1-2\sqrt{\rho})|X|$ vertices of $X$.
On the other hand, for any $x\in X$, consider the link graph $\h_{x}$ of $x$ on $Z$, which contains at least $(1-\rho) \binom {|Z|}2$ edges. By convexity, the number of triples $z_1z_2z_3$ suitable for $x$ is at least
\begin{align*}
\frac13\sum_{z\in Z}\binom{\deg_{\h_{x}}(z)}2 \ge \frac13|Z| \binom{(1-\rho) (|Z|-1)}2 > (1-2\rho)\binom{|Z|}3,
\end{align*}
where the last inequality holds because $|Z|$ is large enough. 
Thus, the number of good triples $z_1z_2z_3$ suitable for $x$ is at least $(1-2\rho-3{\rho})\binom{|Z|}3= (1- 5{\rho})\binom{|Z|}3$.

Let $\F_0$ be the set of good triples in $Z$. We want to form a family of disjoint good triples in $Z$ such that for any $x\in X$, many triples from this family are suitable for $x$. To achieve this, we choose a subfamily $\mathcal F$ from $\F_0$ by selecting each member randomly and independently with probability $p=2\rho^{1/4} |Z|^{-2}$.
Then $|\F|$ follows the binomial distribution $B(|\F_0|, p)$ with expectation $\mathbb{E}(|\F|) = p |\F_0| \le p\binom{|Z|}{3}$.
Furthermore, for every $x\in X$, let $f(x)$ denote the number of members of $\F$ that are suitable for $x$. Then $f(x)$ follows the binomial distribution $B(N, p)$ with $N\ge (1- 5{\rho})\binom{|Z|}3$ by previous calculation. Hence $\mathbb{E}(f(x))\ge p(1- 5{\rho})\binom{|Z|}3$. 
Finally, since there are at most $\binom{|Z|}3\cdot 3\cdot \binom{|Z|-1}2 < \frac14 |Z|^5$ pairs of intersecting triples, the expected number of the intersecting triples in $\mathcal F$ is at most $p^2 \cdot \frac14 |Z|^5= \rho^{1/2}|Z|$.

By applying Chernoff's bound on the first two properties below and Markov's bound on the last one, we can find a family $\mathcal F\subseteq \F_0$ that satisfies
\begin{itemize}
\item
$|\mathcal F|\le 2p \binom{|Z|}3 <\frac23 \rho^{1/4} |Z|$,
\item
for any vertex $x\in X$, at least $\frac12 p\cdot (1- 5{\rho}) \binom{|Z|}3>\frac1{6}\rho^{1/4}(1- 6{\rho})|Z|$ triples in $\mathcal F$ are suitable for $x$,
\item
the number of intersecting pairs of triples in $\mathcal F$ is at most $2\rho^{1/2}|Z|$.
\end{itemize}
After deleting one triple from each of the intersecting pairs from $\F$, we obtain a subfamily $\mathcal F'$ consisting of at most $\frac23 \rho^{1/4} |Z|$ disjoint good triples in $Z$ and for each $x\in X$, at least
\begin{equation}\label{eq:F'}
\frac{\rho^{1/4}}{6}(1- 6{\rho})|Z| - 2\rho^{1/2}|Z|> \frac{ \rho^{1/4}}{12}|Z|
\end{equation}
members of $\mathcal F'$ are suitable for $x$, where the inequality holds because $\rho\le 2\cdot 10^{-6}$.

Denote $\F'$ by $\{Q_1,\dots, Q_q\}$ for some $q\le \frac23 \rho^{1/4} |Z|$. Let $Z_1$ be the set of vertices of $Z$ not in any element of $\F'$. Define $G'=G[Z_1]$. Note that $|Z_1|=|Z|-3q$. For every $v\in Z_1$, $\overline \deg_{G'}(v)\le \overline \deg_G(v)\le \sqrt{\rho} |Z|$ by \eqref{eq:dG}. Thus
\[
\delta(G') \ge |Z_1| - \sqrt{\rho} |Z| =|Z|-3q -\sqrt{\rho} |Z| > \frac{|Z_1|}2,
\]
because $\rho\le 2\cdot 10^{-6}$. By Dirac's Theorem, $G'$ is Hamiltonian. We thus find a Hamiltonian cycle of $G'$, denoted by $b_1 b_2\cdots b_{3m}b_1$, where $m=|X|-q$. Let $Q_{q+i} = b_{3i-2} b_{3i-1} b_{3i}$ for $1\le i \le m$. Then $Q_{q+1}, \dots, Q_{|X|}$ are good triples.

Now consider the bipartite graph $\Gamma$ between $X$ and $\Q := \{Q_1, Q_2,\dots, Q_{|X|}\}$, such that $x\in X$ and $Q_i\in \Q$ are adjacent if and only if $Q_i$ is suitable for $x$. For every $Q_i$, since it is good, $\deg_{\Gamma}(Q_i)\ge (1-2\sqrt{\rho})|X|$. Let $\Q_2=\{Q_{q+1}, \dots, Q_{|X|}\}$. Let $X_0$ be the set of $x\in X$ such that $\deg_{\Gamma}(x, \Q_2)\le |\Q_2|/2$. Then
\[
|X_0|\frac{|\Q_2|}2\le \overline{e}_{\Gamma}(X, \Q_2)\le 2\sqrt{\rho} |X|\cdot |\Q_2|,
\]
which implies that $|X_0|\le 4\sqrt{\rho}|X|=\frac43 \sqrt{\rho}|Z|$.

We now find a perfect matching between $X$ and $\Q$ as follows.
\begin{itemize}
\item[Step 1.]
Each vertex $x\in X_0$ is matched to a different member of $\F'$ that is suitable for $x$ -- this is possible because of \eqref{eq:F'} and $|X_0| \le \frac43 \sqrt\rho |Z| \le \frac1{12}\rho^{1/4} |Z|$ since $\rho\le 2\cdot 10^{-6}$.
\item[Step 2.]
Each of the unused triples in $Q_1 Q_2 \cdots Q_{q}$ is matched to a suitable vertex in $X\setminus X_0$ -- this is possible because $\deg_{\Gamma} (Q_i) \ge (1 - 2\sqrt \rho)|X| \ge q$.
\item[Step 3.]
Let $X_1$ be the set of the remaining vertices in $X$. Then $|X_1|= |X| - q = |\Q_2|$.
Now consider $\Gamma'=\Gamma[X_1 , \Q_2]$. It is easy to check that $\delta(\Gamma')\ge |X_1|/2$ -- thus $\Gamma'$ contains a perfect matching by the Marriage Theorem.
\end{itemize}
The perfect matching between $X$ and $\Q$ gives rise to the desired perfect $\C$-tiling of $\h$ as outlined in the beginning of the proof.
\end{proof}

\section*{Acknowledgment}
We thank Richard Mycroft for helpful discussion on the Regularity Lemma and Extension Lemma. We also thank two anonymous referees for their valuable comments that improved the presentation.

\bibliographystyle{plain}
\bibliography{Nov2013}

\end{document}